\documentclass{amsart}

\usepackage{amsfonts, amsmath, amscd}
\usepackage[psamsfonts]{amssymb}
\usepackage{amssymb}
\usepackage{pb-diagram}
\usepackage[usenames]{color}
\usepackage{hyperref}
\usepackage[all]{xy}
\usepackage{graphicx}

\numberwithin{equation}{section}
\newtheorem{theorem}{Theorem}[section]
\newtheorem{corollary}[theorem]{Corollary}
\newtheorem{proposition}[theorem]{Proposition}
\newtheorem{lemma}[theorem]{Lemma}
\newtheorem{remark}[theorem]{Remark}
\theoremstyle{definition}

\newtheorem{definition}[theorem]{Definition}

\newtheorem{problem}{Problem}

\theoremstyle{plain}

\def\R{\mathbb R}

\DeclareMathOperator*{\supp}{\rm supp}

\DeclareMathOperator{\ldim}{ldim}

\def\cont{\mathfrak c}

\def\hull#1{\langle#1\rangle}

\begin{document}

\title[Products of topological groups]
{Products of topological groups in which all closed subgroups are separable}
\author[Arkady G.~Leiderman and Mikhail G.~Tkachenko]{Arkady G.~Leiderman$\,^{\star,1}$ and Mikhail G.~Tkachenko$\,^{\star,2}$}
\address{(A.~Leiderman) Department of Mathematics, Ben-Gurion University of the Negev, Beer Sheva, P.O.B. 653, Israel}
\email{arkady@math.bgu.ac.il}
\address{(M.~Tkachenko) Departamento de Matem\'aticas, Universidad Aut\'onoma Metropolitana, 
Av. San Rafael Atlixco 186, Col. Vicentina, Del. Iztapalapa, C.P. 09340, Mexico City, Mexico}
\email{mich@xanum.uam.mx}

\thanks{$^\star\,$The authors have been supported by CONACyT of Mexico, grant number 
CB-2012-01 178103.}
\thanks{$^1\,$The first listed author gratefully acknowledges the financial support received 
from the Universidad Aut\'onoma Metropolitana during his visit to Mexico City in September, 2016.} 
\thanks{$^2\,$Corresponding author}


\dedicatory{To the memory of Wistar Comfort (1933--2016),\\ a great topologist and man, 
to whom we owe much of our inspiration}

\date{December 26, 2016}

\begin{abstract}
We prove that if $H$ is a topological group such that all closed subgroups of $H$ 
are separable, then the product $G\times H$ has the same property for every 
separable compact group $G$.

Let $\cont$ be the cardinality of the continuum. Assuming $2^{\omega_1} = \cont$, 
we show that there exist: 
\begin{itemize}
\item pseudocompact topological abelian groups $G$ and $H$ such that all closed 
subgroups of $G$ and $H$ are separable, but the product $G\times H$ contains a 
closed non-separable $\sigma$-compact subgroup;
\item pseudocomplete locally convex vector spaces $K$ and $L$ such that all closed 
vector subspaces of $K$ and $L$ are separable, but the product $K\times L$ contains a
closed non-separable $\sigma$-compact vector subspace.
\end{itemize}
\end{abstract}

\maketitle

\section{Introduction}

All topological groups and locally convex linear spaces are assumed to be Hausdorff.
The \textit{weight} of a topological space $X$, denoted by $w(X)$, is the smallest 
size of a base for $X$. A space $X$ is \emph{separable} if it contains a dense 
countable subset. If every subspace of a topological space $X$ is separable, then 
$X$ is called \emph{hereditarily separable}. Hereditary separability is not a productive 
property\,---\,the Sorgenfrey line is an example of a hereditarily separable paratopological 
group whose square contains a closed discrete subgroup of cardinality $\cont$ (see 
\cite[2.3.12]{Eng89} or \cite[5.2.e]{AT}). Nevertheless, as we observe in Proposition~\ref{Pro:HS}, 
the product of any hereditarily separable topological space with a separable metrizable 
space is hereditarily separable.

Our main objective is to study products of two topological groups having the following 
property: Every \emph{closed subgroup} of a group is separable. Since this property does 
not imply the separability of every subspace of a group, Proposition~\ref{Pro:HS} has very 
limited applicability for our purposes. 

It is known that a closed subgroup of a separable topological group is not necessarily separable.
However, W. Comfort and G. Itzkowitz proved in \cite{ComItz} that all closed subgroups of a 
separable locally compact topological group are separable. It was also noticed by several 
authors independently that every metrizable subgroup of a separable topological group is 
separable (see  \cite{Lohman}).

 Recently these results have been generalized in \cite{LMT} as follows: Every feathered 
 subgroup of a separable topological group is separable. We recall that a topological 
group $G$ is \emph{feathered} if it contains  a compact subgroup $K$ such that the quotient 
space $G/K$ is metrizable (see \cite[Section~4.3]{AT}). All locally compact and all metrizable 
groups are feathered.

Since the class of feathered groups is closed under countable products and taking closed 
subgroups, we obtain the following simple corollary.

\begin{proposition}\label{Pr:feathered}
Let $G$ be a separable locally compact group and $H$ be a separable feathered group.
Then every closed subgroup of the product $G\times H$ is separable.  
\end{proposition}

Let us say that a topological group $G$ is {\it strongly separable} (briefly, \emph{$S$-separable}),
if for any topological group $H$ such that every closed subgroup of $H$ is separable, 
the product $G\times H$ has the same property. 

The following open problem arises naturally.

\begin{problem}\label{Prob:LL} 
Find out the frontiers of the class of $S$-separable topological groups:
\begin{enumerate}
   \item[(a)] Is every separable locally compact group $S$-separable?
   \item[(b)] Is the group of reals $\mathbb{R}$ $S$-separable? 
   Does there exist a separable metrizable group which is not $S$-separable?
   \item[(c)] Is the free topological group on the closed unit  interval 
          $S$-separable?
\end{enumerate} 
\end{problem}

Our Theorem~\ref{motivation} provides the positive answer to (a) of Problem~\ref{Prob:LL}
in the important case when $G$ is a separable compact group.

Then we deduce that every topological group $G$ which contains a separable compact 
subgroup $K$ such that the quotient space $G/K$ is countable, is $S$-separable.

It is reasonable to ask whether the separability of closed 
subgroups of the product $G\times H$ is determined by the same property of the factors 
$G$ and $H$, without imposing additional conditions on $G$ or $H$. We answer this 
question in the negative in Section~\ref{Sec:Groups}. 

A Tychonoff space $X$ is called \emph{pseudocompact} if every continuous real-valued 
function defined on $X$ is bounded. Assuming that $2^{\omega_1}=\cont$, we construct in  
Theorem~\ref{Th:Gr} pseudocompact topological abelian groups $G$ and $H$ such that 
all closed subgroups of $G$ and $H$ are separable, but the product $G\times H$ contains 
a closed non-separable $\sigma$-compact subgroup.

In Section~\ref{Sec:LCS} we consider the class of locally convex spaces (lcs) in which 
all closed vector subspaces are separable. The case of locally convex spaces is quite different from
topological groups, as an infinite-dimensional lcs is never locally compact or pseudocompact. 
Probably the first example of a closed (but not complete) non-separable vector subspace of 
a separable lcs was given by R.~Lohman and W.~Stiles \cite{Lohman}. The study of the products 
of topological vector spaces in which all closed vector subspaces are separable was initiated by  P.~Doma\'nski. He proved in \cite{Domanski} that if $E_i$ is a separable topological vector space whose completion is not $q$-minimal (in particular, if $E_i$ is a separable infinite-dimensional 
Banach space) for each $i\in I$, where $|I|=\cont$, then the product $\prod_{i\in I}E_i$ 
has a non-separable closed vector subspace. 

Recently this result was generalized in \cite{KLM} as follows: If each $E_i$, for $i\in I$, 
is an lcs with at least $\cont$ of the $E_i$'s not having the weak topology, then the product 
$\prod_{i\in I}E_i$ contains a closed non-separable vector subspace.

These facts prompt the following problem for lcs, similar to the questions considered 
earlier for topological groups.

\begin{problem}\label{Prob:LCS}
Do there exist locally convex spaces $K$ and $L$ such that
all closed linear subspaces of $K$ and $L$ are separable, but the product 
$K\times L$ contains a closed non-separable vector subspace?
\end{problem}

To the best of our knowledge, the product of \emph{two} lcs in which all closed vector 
subspaces are separable has not been considered in the literature yet.

 We present a result in the negative direction analogous in spirit to the aforementioned 
Theorem~\ref{Th:Gr}. Our Theorem~\ref{Th:LCS} states that under $2^{\omega_1}=\cont$, 
there exist pseudocomplete (hence, Baire) locally convex spaces $K$ and $L$ such that 
all closed vector subspaces of both $K$ and $L$ are separable, but the product $K\times L$ 
contains a closed non-separable $\sigma$-compact vector subspace.

Note that in view of our Proposition~\ref{findim} none of the factors $K, L$ in 
Theorem~\ref{Th:LCS} can be a finite-dimensional Banach space.

The question whether the assumption $2^{\omega_1}=\cont$ can be dropped in 
Theorems~\ref{Th:Gr} and~\ref{Th:LCS}  remains open (see Problem~\ref{Prob:L1}).  

\subsection{Notation and Background Results}\label{Background}
We start with the following apparently folklore result regarding products
of hereditarily separable topological spaces. The authors thank K.~Kunen who provided 
us with a short proof of the following proposition. Since we failed to find a reference to
this fact in the literature, its proof is included for the sake of completeness.  

\begin{proposition}\label{Pro:HS}
Let $X$ be a hereditarily separable space and  $Y$ a space with a countable network.
Then the product  $X\times Y$ is also hereditarily separable. 
\end{proposition}

\begin{proof}
Let $\mathcal{N}$ be a countable network for $Y$. The space $Y$ admits a finer topology 
with a countable base\,---\,it suffices to consider the topology on $Y$ whose subbase is 
$\mathcal{N}$. Therefore we can assume that the space $Y$ itself has a countable base,
say, $\mathcal{B}$. 

Suppose for a contradiction that the product $X \times Y$ is not hereditarily separable.
Let us recall that a space is hereditarily separable iff it has no uncountable
left separated subspace (see \cite{Roitman}). Let $\{ (x_{\alpha}, y_{\alpha}): 
\alpha < \omega_1 \} $ be a left separated subspace of $X \times Y$, so
there are separating neighborhoods $\{U_{\alpha}: \alpha < \omega_1\}$ 
such that $(x_{\beta}, y_{\beta})\in U_{\beta}$ for each $\beta\in\omega_1$, 
but $(x_{\alpha}, y_{\alpha}) \notin U_{\beta}$ whenever $\alpha < \beta$.
We can assume without  loss of generality that each $U_{\alpha}$ has the form 
$A_{\alpha} \times B_{\alpha}$, where $A_{\alpha}$ is an open subset of $X$ and 
$B_{\alpha}\in\mathcal{B}$. Since $\mathcal{B}$ is countable, one can find an 
uncountable set $I\subset\omega_1$ and an element $B\in\mathcal{B}$ such that 
$B_\alpha=B$ for each $\alpha\in I$. Clearly, $y_{\alpha} \in B$ for each $\alpha\in I$. 
Take $\alpha,\beta\in I$ with $\alpha<\beta$. Then $x_{\beta} \in A_{\beta}$ and 
$x_{\alpha} \notin  A_{\beta}$\,---\,otherwise we would have $(x_\alpha,y_\alpha)\in 
A_\beta\times B=A_\beta\times B_\beta=U_\beta$. This shows that  $\{ x_{\alpha} : 
\alpha\in I\}$ is an uncountable left separated subspace of $X$. Hence $X$ is not 
hereditarily separable, thus contradicting our assumptions.
\end{proof}

Next we collect several important (mostly well-known) facts that will be applied in the sequel. 

\begin{theorem}{\rm (See \cite[Theorem~3.1]{Itz} and \cite[Corollary~2.5]{ComItz})}\label{compact_group}    
 If a compact topological group $G$ satisfies $w(G)\leq\cont$, then it is separable, 
and vice versa. Hence all closed subgroups of a separable compact group $G$ are 
separable. 
\end{theorem}

As usual, we equip products of topological spaces with the Tychonoff topology.
The next result about products of separable spaces follows from the classical 
Hewitt--Marczewski--Pondiczery theorem.

\begin{theorem}{\rm (See \cite[Theorem~2.3.15]{Eng89})}\label{HMP}
The product of no more than $\cont$ separable spaces is separable.
\end{theorem}

Let $X=\prod_{\alpha \in A} X_{\alpha}$ be a product space and $B$ an arbitrary 
non-empty subset of the index set $A$. Then $\pi_B: X \to X_B$ denotes the natural 
projection of $X$ onto the subproduct $X_B =\prod_{\alpha \in B} X_{\alpha}$.

We will use the following theorem about subspaces of Tychonoff products of compact 
metrizable spaces in the proof of Theorem~\ref{Th:Gr}.

\begin{theorem}\label{projections1} \cite[Theorem~2.4.15]{AT}
Suppose that $S$ is a subspace of the topological product $X=\prod_{\alpha \in A} X_{\alpha}$ 
of compact metrizable spaces such that $\pi_B(X)=X_B$ for every countable subset $ B \subset A$.
Then $X$ is the Stone-\v{C}ech compactification of $S$ and the space $S$ is pseudocompact. 
\end{theorem}

The closure of a subset $U \subset X$ in $X$ is denoted by $cl_X {U}$.
A family $\mathcal{P}$ of open non-empty subsets of $X$ is called a \emph{$\pi$-base} 
for $X$ if every open non-empty set in $X$ contains an element of $\mathcal{P}$. 
The following notion was introduced by J.C.~Oxtoby in \cite{Ox}.

\begin{definition}\label{pseudocomplete}
A topological space $X$ is called \emph{pseudocomplete}
if there exists a sequence $\{\mathcal{P}_n: n\in\omega\}$ of $\pi$-bases
for $X$ such that for every sequence $\{U_n: n\in\omega\}$ of subsets of $X$ satisfying $U_n\in\mathcal{P}_n$ and $cl_X {U}_{n+1}\subset U_n$ for each $n\in\omega$, the 
intersection $\bigcap_{n\in\omega}U_n$ is non-empty.
\end{definition}

Every regular pseudocomplete space has the Baire property, and an arbitrary product of pseudocomplete spaces is pseudocomplete \cite{Ox}.
 The class of pseudocomplete spaces contains all pseudocompact spaces.

Several well-known problems about pseudocompleteness are still open.
For example, it is unknown whether pseudocompleteness is preserved by continuous 
open mappings or is hereditary with respect to dense $G_{\delta}$-subspaces (see 
\cite{Mill}). Let us say that a subset $Y$ of a space $X$ is \emph{$G_{\delta}$-dense} 
in $X$ if $Y$ intersects each non-empty $G_{\delta}$-set in $X$. The following fact is 
apparently new though simple. 

We recall that a space $X$ is \emph{Moscow} if the
closure of every open set in $X$ is the union of a family of $G_\delta$-sets in $X$.
All extremely disconnected spaces and all spaces of countable pseudocharacter are
evidently Moscow.

\begin{proposition}\label{dense_new}
Let $X$ be a regular pseudocomplete Moscow space and $Y$ be a $G_{\delta}$-dense 
subspace of $X$. Then $Y$ is a pseudocomplete space as well.
\end{proposition}

\begin{proof} 
Fix a sequence $\{\mathcal{P}_n: n\in\omega\}$ of $\pi$-bases for $X$ witnessing the
pseudocompleteness of $X$. Since $X$ is regular, we may assume that each $\mathcal{P}_n$ 
consists of regularly open sets. For every $n\in\omega$, we put $\mathcal{Q}_n = 
\{U \cap Y: U \in \mathcal{P}_n\}$. We claim that the sequence $\{\mathcal{Q}_n: n\in\omega\}$ 
satisfies the requirements of Definition~\ref{pseudocomplete}.

Indeed, let $\{W_n: n\in\omega\}$ be a sequence such that $W_n\in \mathcal{Q}_n$ and
$cl_Y{W_{n+1}}\subset W_n$ for each $n\in\omega$. For every $n\in\omega$, take an open
set $U_n\in\mathcal{P}_n$ with $U_n\cap Y=W_n$. It is easy to see that $cl_X{U_{n+1}}\subset
U_n$ for all $n\in\omega$. If not, then $cl_X{U_{n+1}}\not\subset U_n$ for some $n\in\omega$.
Since the set $U_n$ is regular open in $X$, the latter means that 
$$
cl_{X}{U_{n+1}}\cap cl_X(X\setminus cl_X{U_n})\neq\emptyset.
$$
As $X$ is a Moscow space, each of the sets $cl_{X}{U_{n+1}}$ and $cl_X(X\setminus cl_X{U_n})$
is the union of $G_\delta$-sets. By the $G_\delta$-density of $Y$ in $X$, we conclude that
\begin{equation}\label{eq:1}
 Y\cap \big(cl_{X}{U_{n+1}}\cap cl_X(X\setminus cl_X{U_n})\big)\neq\emptyset
\end{equation}
or, equivalently, $cl_Y{W_{n+1}}\cap cl_X(X\setminus cl_X{U_n})\neq\emptyset$.
However, $cl_Y{W_{n+1}}\subset W_n\subset U_n$ and $U_n\cap cl_X(X\setminus cl_X{U_n})=\emptyset$, which contradicts (\ref{eq:1}). This proves that $cl_X{U_{n+1}}\subset
U_n$ for each $n\in\omega$.

Since $U_n\in\mathcal{P}_n$ for each $n\in\omega$ and  the space $X$ is pseudocomplete,
it follows that $\bigcap_{n\in\omega} U_n\neq\emptyset$. Making use of the 
$G_\delta$-denseness of $Y$ in $X$, we see that 
$$
\emptyset\neq Y\cap\bigcap_{n\in\omega} U_n=\bigcap_{n\in\omega} (U_n\cap Y)=
\bigcap_{n\in\omega} W_n.
$$
This implies the pseudocompleteness of $Y$.
\end{proof}

Let us recall that the \emph{$o$-tightness} of a space $X$, denoted by $ot(X)$, is the 
minimum cardinal $\kappa\geq\omega$ such that for every family $\gamma$ of open 
sets in $X$ and every point $x\in \overline{\bigcup\gamma}$, one can find a subfamily 
$\lambda$ of $\gamma$ with $|\lambda|\leq\kappa$ such that $x\in \overline{\bigcup\lambda}$. 
It is clear that every space $X$ satisfies $ot(X)\leq c(X)$ and $ot(X)\leq t(X)$, where $c(X)$
and $t(X)$ are the cellularity and tightness of $X$, respectively (see \cite{Tk83}). 

In the presence of an additional algebraic structure on a given space $X$, mild
topological restrictions on $X$, like having countable $o$-tightness, imply that
$X$ is a Moscow space (see \cite[Section~6.4]{AT}). We apply this fact in the
following corollary.

\begin{corollary}\label{Cor:Parato}
Let $G$ be a regular paratopological group of countable $o$-tightness. If $G$ is
pseudocomplete, then so is every $G_\delta$-dense subspace of $G$.
\end{corollary}

\begin{proof}
Since $G$ has countable $o$-tightness, \cite[Corollary~6.4.11,\,5)]{AT} implies that $G$ 
is a Moscow space. Hence the required conclusion follows from Proposition~\ref{dense_new}.
\end{proof}

The next result will be used in the proof of Theorem~\ref{Th:LCS}. 

\begin{theorem}\label{projections2}
Let $Y$ be a subspace of the topological product $X=\prod_{\alpha \in A} X_{\alpha}$ of 
regular pseudocomplete first countable spaces such that $\pi_B(Y)=\prod_{\alpha\in B} 
X_\alpha$ for every countable subset $B$ of $A$. Then the space $Y$ is pseudocomplete.
\end{theorem}

\begin{proof} 
It is clear that $Y$ is a $G_{\delta}$-dense subspace of $X$ because $Y$ fills all countable 
faces of the product space $X$. Also, the space $X$ is pseudocomplete as a product of  pseudocomplete spaces \cite{Ox}. Since each factor $X_\alpha$ is regular and first countable, 
it follows from \cite[Corollary~6.3.15]{AT} that $X$ is a regular Moscow space. Finally, $Y$ is pseudocomplete in view of Proposition~\ref{dense_new}.
\end{proof}

If the factors $X_\alpha$ are paratopological groups, we can complement
Theorem~\ref{projections2} as follows.

\begin{corollary}\label{Cor:SepPG}
Let $Y$ be a subspace of the topological product $H=\prod_{\alpha \in A} H_{\alpha}$ of 
regular, pseudocomplete, separable paratopological groups. If $\pi_B(Y)=\prod_{\alpha\in B} 
H_\alpha$ for every countable subset $B$ of $A$, then the space $Y$ is pseudocomplete.
\end{corollary}

\begin{proof}
By \cite[Theorem~6.4.19]{AT}, $H$ is a Moscow space. Since $Y$ is $G_\delta$-dense
in $H$ and $H$ is regular, it remains to apply Proposition~\ref{dense_new}.
\end{proof}

Every maximal linearly independent subset $B$ of a vector space $E$ is called a 
\emph{Hamel basis} for $E$. The cardinality of $B$ is an algebraic dimension of $E$ 
which will be denoted by $\ldim (E)$. It is known that $\ldim (E)= \cont$ for any separable 
infinite-dimensional Banach space $E$ (see \cite{Lacey}).

\section{Products with a compact or countable factor}\label{Sec:C-compact_Groups}

Let us say that a topological group $G$ is \emph{strongly separable} (briefly, 
\emph{$S$-separable}) if for any topological group $H$ such that every closed 
subgroup of $H$ is separable, the product $G\times H$ has the same property. 

One of our main observations is the following result which can be reformulated by saying
that every separable compact group is $S$-separable. 

\begin{theorem}\label{motivation}
Let $G$ be a separable compact group and $H$ be a topological group in which all 
closed subgroups are separable. Then all closed subgroups of the product $G\times H$ 
are separable as well.
\end{theorem}

\begin{proof}
Take a closed subgroup $C$ of $G\times H$ and denote by $p_H$ the projection 
of $G\times H$ onto the second factor. According to Kuratowski's theorem (see 
\cite[Theorem~3.1.16]{Eng89}) $p_H$ is a closed mapping. Therefore  the image 
$D=p_H(C)$ is a closed subgroup of $H$. It follows from our assumptions about 
$H$ that the group $D$ is separable. Let $\pi_H$ be the restriction of $p_H$ to 
$C$ and $K$ be the kernel of $\pi_H$. Clearly the homomorphism $\pi_H\colon C\to D$ 
is a continuous closed mapping. Hence the homomorphism $\pi_H$ of $C$ onto 
$D$ is a quotient mapping and therefore $\pi_H$ is open \cite[Proposition~1.5.14]{AT}. 
The group $K$ is topologically isomorphic to a closed subgroup of $G$, so $K$ is 
separable according to Theorem~\ref{compact_group}. Finally, $C$ is separable 
because separability is a three-space property in topological 
groups \cite[Theorem~1.5.23]{AT}.
\end{proof}

It is not clear to which extent one can generalize Theorem~\ref{motivation} by weakening 
the compactness assumption on $G$. However, some additional conditions on the groups 
$G$ and/or $H$ have to be imposed as it follows from Theorem~\ref{Th:Gr} in 
Section~\ref{Sec:Groups}. 

In the next proposition we present another situation when the projection 
$G\times H\to H$ turns out to be a closed mapping.

\begin{proposition}\label{Pro:CC}
Let $G$ be a countably compact topological group and $H$ a separable
metrizable topological group. If all closed subgroups of $G$ are separable,
then the product group $G\times H$ has the same property.
\end{proposition}

\begin{proof}
It is known (see \cite[Theorem~3.10.7]{Eng89}) that the projection $p\colon G\times H\to H$
is a closed mapping. Let $C$ be a closed subgroup of $G\times H$ and $\pi$ be 
the restriction to $C$ of the projection $p$. Since $C$ is closed in $G\times H$,
$\pi$ is also a closed mapping. The mapping $\pi$ being a continuous homomorphism, 
we see that $\pi\colon C\to \pi(C)$ is open. Now we finish the proof by the same argument
as in Theorem~\ref{motivation}.
\end{proof}

The following problem arises in an attempt to generalize Proposition~\ref{Pro:CC}:

\begin{problem}\label{Prob:CN}
Let $G$ be a countably compact topological group such that all closed subgroups 
of $G$ are separable, and $H$ a topological group with a countable network. 
Are the closed subgroups of $G\times H$ separable?
\end{problem}

Next we show that every countable topological group is $S$-separable. A more
general result will be presented in Theorem~\ref{class_S}.
 
\begin{proposition}\label{countable}
Let $G$ be a countable topological group and $H$ be a topological group in which 
all closed subgroups are separable. Then all closed subgroups of the product 
$G\times H$ are separable as well.
\end{proposition}

\begin{proof}
We modify slightly the idea presented in the proof of Theorem~\ref{motivation}.
Take a closed subgroup $C$ of $G\times H$ and let $\pi$ be the restriction 
to $C$ of the projection $G\times H\to G$. Then the image $D=\pi(C)$ is a 
countable subgroup of $G$. The kernel of $\pi$ is topologically isomorphic 
to a closed subgroup of $H$ and, hence, is separable. Therefore all fibers of 
$\pi$ are separable. For every $y\in D$, let $S_y$ be a countable dense 
subset of $\pi^{-1}(y)$. Then $S=\bigcup_{y\in D} S_y$ is a countable dense
subset of $C$. Indeed, let $U$ be an arbitrary non-empty open set in $C$. Take
an element $x\in U$ and put $y=\pi(x)$. Then $x\in U\cap \pi^{-1}(y)\neq\emptyset$,
so the density of $S_y$ in $\pi^{-1}(y)$ implies that $U\cap S_y\neq\emptyset$.
Since $S_y\subseteq S$, we conclude that $U\cap S\neq\emptyset$, which
shows that $S$ is dense in $C$. Hence $C$ is separable.
\end{proof}

\begin{proposition}\label{operations}
The class of $S$-separable groups is closed under the operations:
\begin{enumerate}
\item[{\rm (1)}] finite products; 
\item[{\rm (2)}] taking closed subgroups; 
\item[{\rm (3)}] taking continuous homomorphic images.
\end{enumerate}
\end{proposition}

\begin{proof}
Items (1) and (2) are evident, so we verify only (3). Let $\phi\colon F\to G$ be a continuous
onto homomorphism of topological groups, where the group $F$ is $S$-separable. Also,
let $H$ be a topological group such that all closed subgroups of $H$ are separable. 
Denote by $i_H$ the identity mapping of $H$ onto itself. Then $g= \phi \times i_H$
is a continuous homomorphism of $F\times H$ onto $G\times H$. If $D$ is a closed 
subgroup of  $G\times H$, then $C=g^{-1}(D)$ is a separable closed subgroup of 
$F\times H$ since $F$ is $S$-separable. Hence the group $D=g(C)$ is separable as well.
\end{proof}

Denote by $\mathfrak{S}$ the smallest class of topological groups which is generated 
by all compact separable groups, all countable groups and is closed under the 
operations listed in (1)--(3) of Proposition~\ref{operations}. It is not difficult to verify 
that if $G\in \mathfrak{S}$, then $G$ contains a compact separable subgroup $K$ 
such that the quotient space $G/K$ is countable. In the next problem we conjecture 
that this property characterizes the groups from $\mathfrak{S}$:
 
\begin{problem}\label{intr}
Is it  true that a topological group $G$ is in the class $\mathfrak{S}$ if and only if $G$ 
contains a compact separable subgroup $K$ such that the quotient space $G/K$ is 
countable?
\end{problem}

The theorem below generalizes both Theorem~\ref{motivation} and Proposition~\ref{countable}. 
It can be considered as a partial positive answer to Problem~\ref{intr}.

\begin{theorem}\label{class_S}
A topological group $G$ is $S$-separable provided it contains a separable compact 
subgroup $K$ such that the quotient space $G/K$ is countable.
\end{theorem}

\begin{proof}
Consider an arbitrary topological group $H$ such that all closed subgroups of $H$ 
are separable. Let $C$ be a closed subgroup of $G\times H$. It follows from 
Theorem~\ref{motivation} that the closed subgroup $F=(K\times H)\cap C$ of 
$K\times H$ is separable. Let $p\colon G\times H\to G$ be the projection onto 
the first factor. Take any point $x\in p(C)$ and choose an element $z=(x,h)\in C$. 
It is easy to see that $(xK\times H)\cap C=zF$. Since $F$ has countable index
in $G$, the latter equality implies that the group $C$ can be covered by countably 
many translates of the separable group $F$. Hence $C$ is separable as well.
We conclude therefore that $G$ is $S$-separable. 
\end{proof}

\begin{remark}
{\rm Each $G \in \mathfrak{S}$ is a separable $\sigma$-compact group, but
the group of reals $\mathbb{R}$ is not in the class $\mathfrak{S}$.
We do not know any example of an $S$-separable topological group 
which is not in the class $\mathfrak{S}$.
    
The main obstacle for resolving Problem~\ref{Prob:LL} is the fact that
the restriction of an open continuous homomorphism to a closed subgroup
can fail to be open, even if the restriction is considered as a mapping onto 
its image. This is an important issue since we use the fact that separability 
is a three-space property, while the corresponding homomorphism of a group 
onto its quotient group is open. We also note that there exists a continuous
one-to-one homomorphism of a non-separable precompact group onto a 
separable metrizable group (one can combine Theorems~9.9.30 and~9.9.38 
of \cite{AT}). In particular, the kernel of such a homomorphism is trivial and, 
hence, separable. So the preservation of separability under taking inverse 
images of a continuous homomorphism with a separable kernel depends 
essentially on whether the homomorphism is open or not.}
\end{remark}

A topological group $G$ is called {\it categorically compact} (briefly, $C$-compact),
if for every topological group $H$ the projection $G\times H \to H$ sends closed 
subgroups of $G\times H$ to closed subgroups of $H$ \cite{DikUsp}. It is known 
that $C$-compactness is preserved by continuous surjective homomorphisms and 
inherited by closed subgroups. 
D.~Dikranjan and V.~Uspenskij proved that the product of any family of $C$-compact 
groups is $C$-compact. A countable discrete group is $C$-compact if and only if it is 
hereditarily non-topologizable \cite{Lukacs}. Obviously, compact groups are $C$-compact 
and $C$-compactness of $G$ yields its compactness provided that the group $G$ is 
either soluble (in particular, abelian), or connected, or locally compact \cite{DikUsp}.

The long-standing problem of whether every $C$-compact group is compact has been 
recently resolved negatively in the article \cite{KOO}, where an infinite discrete $C$-compact
group is presented. Clearly this group is far from commutative or soluble. Thus, $C$-compact 
groups constitute a rich non-trivial class containing all compact groups as a proper subclass.

\begin{remark}
{\rm We do not know whether all separable $C$-compact topological groups are $S$-separable.}
\end{remark}

\section{Product of Two Pseudocompact Groups}\label{Sec:Groups}
In this section we present two pseudocompact abelian groups $G$ and $H$ such that
all closed subgroups of $G$ and $H$ are separable, but the product $G\times H$
contains a closed non-separable subgroup. 

First we recall that a {\em Boolean group} is a group in which all elements are of order 
two. Clearly, all Boolean groups are abelian. For each integer $n\geq 2$, $\mathbb{Z}(n)$ 
denotes the discrete group $\{0,1,\ldots,n-1\}$ with addition modulo $n$.
A non-empty subset $X$ of a Boolean group $G$ with identity 
$e$ is \emph{independent} if for any pairwise distinct elements $x_1,\ldots,x_n$ of $X$
 the equality $x_1+ \cdots + x_n=e$ implies that $x_1=\cdots= x_n=e$.

A family $\mathcal{V}$ of non-empty subsets of a topological space $X$ is
called a \emph{$\pi$-network} for $X$ if every non-empty open set $U \subset X$ 
contains an element of $\mathcal{V}$.

\begin{lemma}\label{pinet_groups}
Let $\kappa$ be a cardinal satisfying $\omega\leq\kappa\leq\cont$. Then the compact 
Boolean group $C=\mathbb{Z}(2)^\kappa$ has a countable $\pi$-network $\mathcal{V}=
\{V_n: n\in\omega\}$ such that $|V_n|\geq 2^\omega$ for each $n\in\omega$.
\end{lemma}

\begin{proof} 
Identify $\kappa$ with a dense subset of the open interval $(0,1)$ and fix a countable 
family $\mathcal{T}$ consisting of the sets of the form $A\cap \kappa$, with $A$ being 
a disjoint finite union of open intervals with rational  end-points in $(0,1)$. For every set 
$A_1\cup A_2\cup\cdots \cup A_n \in \mathcal{T}$ and a finite collection $\{B_1, B_2,\ldots, 
B_n\}$, where each $B_i = \{0\}\,\, \mbox{or} \,\,\{1\}$, we define the set
$$
V= \{x\in \Pi: x(\alpha) \in B_i \,\, \mbox{for each} \,\, \alpha \in A_i\}.
$$
It is easy to verify that the family $\mathcal{V}$ consisting of all such sets $V$ is a 
countable $\pi$-network for the space $C$. The cardinality of each $V_n$ is at least 
$2^\omega=\cont$ because $V_n$ contains a copy of $\mathbb{Z}(2)^\omega$.
\end{proof}

\begin{proposition}\label{Pro:1}
Let $\kappa$ be a cardinal satisfying $\omega\leq\kappa\leq\cont$ and $S$ be a subgroup 
of the compact Boolean group $C=\mathbb{Z}(2)^\kappa$ with $|S|<\cont$. Then $C$ contains 
a countable dense independent subset $X$ of $C$ such that $\hull{X}\cap S=\{e\}$.
\end{proposition}

\begin{proof}
Evidently, if $x\in C\setminus S$, then $\hull{x}\cap S=\{e\}$. Let 
$\mathcal{V} =\{V_n: n\in\omega\}$ be a countable $\pi$-network for $C$ such that 
every $V_n$ has cardinality at least $2^\omega=\cont$ (see Lemma~\ref{pinet_groups}).
Take an element $x_0\in V_0\setminus S$. Then $\hull{x_0}\cap S=\{e\}$. Similarly, take 
an element $x_1\in V_1\setminus (S+\hull{x_0})$. Again, this is possible since 
$|S+\hull{x_0}|<\cont$. In general, if elements $x_0,x_1,\ldots,x_{n-1}$ of $C$ have been 
defined, we choose an element $x_{n}\in V_{n}\setminus (S+\hull{x_0,x_1,\ldots,x_{n-1}})$. 
This choice guarantees that $\hull{x_0,x_1,\ldots,x_{n}}\cap S=\{e\}$
and that the set $\{x_0,x_1,\ldots,x_n\}$ is independent.

Let $X=\{x_n: n\in\omega\}$ and $Q$ be the subgroup of $C$ generated by $X$. 
Notice that the set $X$ is independent. Since $x_n\in V_n$ for each $n\in\omega$, 
we see that $X$ is dense in $C$. It is also clear that $Q\cap S = \{e\}$. This completes 
the proof. 
\end{proof}

\begin{remark}
{\rm Proposition~\ref{Pro:1} cannot be extended to compact metrizable bounded torsion 
groups. Indeed, let $G=\mathbb{Z}(2)^\omega\times\mathbb{Z}(4)$. Clearly $G$ is a compact
metrizable group of period $4$. Let $S=\{\bar{0}\}\times\mathbb{Z}(4)$, where $\bar{0}$ is 
the identity element of $\mathbb{Z}(2)^\omega$. Then every dense subgroup $D$ of $G$ 
has a non-trivial intersection with the finite group $S$. To see this, consider the open subset
  $U=\mathbb{Z}(2)^\omega\times\{1\}$ of $G$. Since $D$ is dense in $G$, there exists an 
element $x\in D\cap U$. Clearly the element $2x$ is distinct from the identity of $G$ and 
$2x=(\bar{0},2)\in D\cap S$.}
\end{remark}

\begin{theorem}\label{Th:Gr}
Assume that $2^{\omega_1}=\cont$. Then there exist pseudocompact abelian topological 
groups $G$ and $H$ such that all closed subgroups of $G$ and $H$ are separable, but 
the product $G\times H$ contains a closed non-separable $\sigma$-compact subgroup.
\end{theorem}

\begin{proof}
We will construct $G$ and $H$ as dense subgroups of the compact Boolean 
group $\Pi=\mathbb{Z}(2)^{\omega_1}$. For every $\alpha\in\omega_1$, we denote by 
$p_\alpha$ the projection of $\Pi$ onto the $\alpha$-th factor, $p_\alpha(x)=x(\alpha)$ 
for each $x\in \Pi$. Given an element $x\in\Pi$, let 
$$
\supp(x)=\{\alpha\in\omega_1: p_\alpha(x)=1\}.
$$ 
Then the set
$$
\sigma=\{x\in\Pi: |\supp(x)|<\omega\}
$$
is a dense subgroup of $\Pi$ satisfying $|\sigma|=\omega_1$. It is easy to verify that
the group $\sigma$ with the topology inherited from $\Pi$ is $\sigma$-compact and not 
separable.

Our aim is to define the subgroups $G$ and $H$ of $\Pi$ satisfying the equality 
$G\cap H=\sigma$. It is clear that $\Delta=\{(x,x): x\in \Pi\}$, the diagonal in 
$\Pi\times \Pi$, is a closed subgroup of $\Pi\times \Pi$, so $\Delta\cap (G\times H)$ 
is a closed non-separable subgroup of $G\times H$ which is isomorphic to $\sigma$.

We define the groups $G$ and $H$ by recursion of length $\cont$. It follows from 
$2^{\omega_1}=\cont$ that the family of all closed subsets of $\Pi$ has cardinality 
$2^{\omega_1}=\cont$. Hence we can enumerate all infinite closed subgroups of 
$\Pi$, say, $\{C_\alpha: \alpha<\cont\}$. For every countable subset $B$ of $\omega_1$, 
the set $\mathbb{Z}(2)^B$ has cardinality at most $\cont$, so we can enumerate the 
set $\Sigma=\bigcup\{\mathbb{Z}(2)^B: B\subset\omega_1,\ 1\leq |B|\leq\omega\}$ 
as $\Sigma=\{b_\alpha: \alpha<\cont\}$. For every $\alpha<\cont$, let $B_\alpha$ be
a countable subset of $\omega_1$ such that $b_\alpha\in \mathbb{Z}(2)^{B_\alpha}$. 
The two enumerations will be used in our construction of $G$ and $H$. For every
non-empty subset $B$ of $\omega_1$, we denote the projection of 
$\Pi=\mathbb{Z}(2)^{\omega_1}$ onto $\mathbb{Z}(2)^B$ by $\pi_B$.

We start with putting $G_0=H_0=\sigma$. Let $\alpha$ be an ordinal, $0<\alpha<\cont$.
Assume that we have defined subgroups $G_\beta$ and $H_\beta$ of $\Pi$, for each 
$\beta<\alpha$, such that the following conditions hold:
\begin{enumerate}
\item[(i)]   $G_\gamma\subset G_\beta$ and $H_\gamma\subset H_\beta$ if $\gamma<\beta$;
\item[(ii)]  $|G_\beta|\leq |\beta+1|\cdot\omega_1$ and $|H_\beta|\leq |\beta+1|\cdot\omega_1$; 
\item[(iii)] $G_\beta\cap H_\beta=\sigma$;
\item[(iv)] $b_\gamma\in \pi_{B_\beta}(G_\beta)\cap \pi_{B_\beta}(H_\beta)$ for each 
                $\gamma<\beta$;
\item[(v)]  both $G_\beta\cap C_\gamma$ and $H_\beta\cap C_\gamma$ contain 
                a countable dense subgroup of $C_\gamma$, for each $\gamma<\beta$.
\end{enumerate}
If $\alpha$ is a limit ordinal, we put $G_\alpha=\bigcup_{\beta<\alpha} G_\beta$ 
and $H_\alpha=\bigcup_{\beta<\alpha} H_\beta$. It is clear that the families 
$\{G_\beta: \beta\leq\alpha\}$ and $\{H_\beta: \beta\leq\alpha\}$ satisfy conditions
(i)--(iv). 

Assume that $\alpha$ is a successor ordinal, say, $\alpha=\nu+1$. First we define 
a subgroup $G_\alpha$ of $\Pi$. It follows from (ii) that $|G_\nu|\leq |\nu|\cdot\omega_1$ 
and $|H_\nu|\leq |\nu|\cdot\omega_1$.
It is known that every compact Boolean group is topologically isomorphic to
the group $\mathbb{Z}(2)^\lambda$ for some cardinal $\lambda$ (this is a simple 
corollary of the Pontryagin--Van Kampen's duality theory, see \cite[Chapter 5]{Morris}).
 Hence one can apply Proposition~\ref{Pro:1} with $S=G_\nu + H_\nu$ to find a 
 countable dense subgroup $Q_\nu$ of a compact Boolean group $C_\nu$ such 
that the intersection of $Q_\nu$ and $S$ is trivial. This implies the equality
$$
(G_\nu+Q_\nu)\cap (H_\nu+Q_\nu) = G_\nu \cap H_\nu = \sigma. 
$$
Let $G'_\nu=G_\nu+Q_\nu$ and $H'_\nu=H_\nu+Q_\nu$. By (ii), we have that 
$|G'_\nu|\leq |\nu+1|\cdot \omega_1$ and $|H'_\nu|\leq |\nu+1|\cdot \omega_1$.
Since $Q_\nu\subset G'_\nu\cap H'_\nu$, both intersections $G'_\nu\cap C_\nu$
and $H'_\nu\cap C_\nu$ contain the countable dense subgroup $Q_\mu$ of $C_\nu$.
Denote by $P_\nu$ the set $\{x\in\Pi: \pi_{B_\nu}(x)=b_\nu\}$. Then $|P_\nu|=\cont$,
while $|G'_\nu|\cdot|H'_\nu|<\cont$. Hence we can choose an element $x_\nu\in P_\nu$ 
such that $x_\nu\notin G'_\nu + H'_\nu$. We put $G_\alpha=G'_\nu + \hull{x_\nu}$. It
follows from our choice of $x_\nu$ that $G_\alpha\cap H'_\nu=\sigma$. Similarly,
one can choose $y_\nu\in P_\nu$ such that $y_\nu\notin G_\alpha + H'_\nu$, and
we put $H_\alpha=H'_\nu + \hull{y_\nu}$. Again, our choice of $y_\nu$ implies that
$G_\alpha\cap H_\alpha=\sigma$ and, clearly, $b_\nu\in \pi_{B_\nu}(G_\alpha)
\cap \pi_{B_\nu}(H_\alpha)$. Therefore, the families $\{G_\beta: \beta\leq\alpha\}$ 
and $\{H_\beta: \beta\leq\alpha\}$ satisfy conditions (i)--(v). 

Finally we define subgroups $G$ and $H$ of $\Pi$ by letting $G=\bigcup_{\alpha<\cont} G_\alpha$
and $H=\bigcup_{\alpha<\cont} H_\alpha$. Then (i) and (iii) together imply that $G\cap H=\sigma$. 
We claim that all closed subgroups of $G$ and $H$ are separable. Clearly it suffices to verify
this only for $G$. Let $F$ be an infinite closed subgroup of $G$. Then the closure of $F$ in
$\Pi$, say, $\overline{F}$ is an infinite closed subgroup of $\Pi$, so $\overline{F}=C_\alpha$,
for some $\alpha<\cont$. It follows from (v) that $G_{\alpha+1}\cap C_\alpha$ contains a
countable dense subgroup of $C_\alpha$, and so does $G\cap C_\alpha=G\cap \overline{F}=F$.
Therefore the group $F$ is separable, as claimed.

It remains to show that the groups $G$ and $H$ are pseudocompact. Let us recall that
$\{b_\nu: \nu<\cont\}$ is an enumeration of all countable subproducts in $\Pi=\mathbb{Z}(2)^{\omega_1}$, so (iv) implies that $\pi_B(G)=\pi_B(H)=D^B$ for each non-empty countable 
set $B\subset\omega_1$. Since $\mathbb{Z}(2)$ is a compact metrizable group, the 
pseudocompactness of both $G$ and $H$ follows directly from Theorem~\ref{projections1}.
\end{proof}

\section{Product of Two Pseudocomplete LCS}\label{Sec:LCS}
First, we present a result similar to Proposition~\ref{countable}.

\begin{proposition}\label{findim}
Let $K$ be a finite-dimensional Banach space and $L$ be a topological vector space 
in which all closed vector subspaces are separable. Then all closed vector subspaces 
of the product $K\times L$ are separable as well.
\end{proposition}

\begin{proof}
Take a closed vector subspace $C$ of $K\times L$ and let $\pi$ be the restriction 
of the projection $K\times L\to K$ to $C$. Mapping $\pi$ is linear, hence the image 
$D=\pi(C)$ is a finite-dimensional Banach space. It is widely known that any continuous 
linear mapping onto a finite-dimensional Banach space is an open mapping (see \cite{Rudin}). 
So, $\pi\colon C\to D$ is a linear open mapping onto the  separable space $D$. Again, 
the kernel of $\pi$ is linearly isomorphic to a closed vector subspace of $L$, therefore 
the kernel of $\pi$ is separable, by the assumptions about $L$.
Finally, $C$ is also separable because separability is a three-space property. 
\end{proof}

\begin{remark}
{\rm We do not know whether Proposition~\ref{findim} remains valid for
an arbitrary separable Banach space $K$.}
\end{remark}

We show in Theorem~\ref{Th:LCS} below that the answer to Problem~\ref{Prob:LCS} 
is \lq\lq{Yes\rq\rq} under the assumption that $2^{\omega_1}=\cont$. In other words,
we present locally convex spaces $K$ and $L$ such that all closed vector subspaces
of $K$ and $L$ are separable, but the product $K\times L$ contains a closed non-separable
vector subspace. In addition, the spaces $K$ and $L$ are pseudocomplete.

First we establish two auxiliary facts about the properties 
of $\R^{\kappa}$, where $\omega\leq\kappa\leq\cont$. They are analogous to 
Lemma~\ref{pinet_groups} and Proposition~\ref{Pro:1} and will play a similar role.

\begin{lemma}\label{prop:pinet}
Let $\kappa$ be a cardinal satisfying $\omega\leq\kappa\leq\cont$. Then
the space $\Pi = \R^{\kappa}$ has a countable $\pi$-network $\mathcal{V}$ 
such that every element $V\in\mathcal{V}$ contains a linearly independent 
subset of size at least $\cont$.
\end{lemma}

\begin{proof} 
Identify $\kappa$ with a dense subset of the open interval $(0,1)$ and fix two countable 
families $\mathcal{B}$ and $\mathcal{T}$, where $\mathcal{B}$ consists of open intervals 
with rational end-points in $\R$ and $\mathcal{T}$ consists of the sets of the form 
$A\cap \kappa$, with $A$ being a disjoint finite union of open intervals with rational 
end-points in $(0,1)$. For every set $A_1\cup A_2\cup\cdots \cup A_n \in \mathcal{T}$ 
and a finite collection $\{B_1, B_2,\ldots, B_n\}$, where each $B_i \in \mathcal{B}$, we define the set
$$
V= \{x\in \Pi: x(\alpha) \in B_i \,\, \mbox{for each} \,\, \alpha \in A_i\}.
$$
It is easy to verify that the family $\mathcal{V}$ consisting of all such sets $V$ is a 
countable $\pi$-network for the space $\Pi$. 

To finish the proof we make use of an idea from \cite{Lacey}. Consider an arbitrary 
element $V\in\mathcal{V}$. Since $\kappa$ is infinite there exists a subset $W \subset V$ 
which is linearly isomorphic to the countable product $[a,b]^{\omega}$, where $[a,b]$ 
is a closed segment in $\R$. Without loss of generality we can assume that $[a,b] = [0,1]$. 
It suffices to find  a linearly independent subset of size $\cont$ in $W = [0,1]^{\omega}$ 
considered as a linear subspace of $\mathbb{R}^\omega$.  Let $\{N_t: t \in I\}$ be an 
almost disjoint family consisting of infinite subsets of $\omega$ and such that $|I| = \cont$. 
For every $t \in I$, let $x_t$ be an element of $W$ which is defined by $x_t(n) = 1$ if 
$n \in N_t$ and $x_t(n) = 0$ otherwise. It is easy to see that the family 
$\{x_t: t \in I\} \subset W$ is linearly independent.
\end{proof}

\begin{proposition}\label{pro:Lin}
Let $\kappa$ be an infinite cardinal with $\kappa\leq\cont$ and $L,S$ be vector subspaces
of $\mathbb{R}^\kappa$. If $L$ is closed, $\ldim(L)\geq\omega$ and $\ldim(S) < \cont$, then 
$L$ contains a dense vector subspace $M$ such that $\ldim(M) = \omega$ and $M\cap S=\{0\}$.
\end{proposition}

\begin{proof}
Every closed vector subspace of 
$\mathbb{R}^\kappa$ is topologically isomorphic to $\mathbb{R}^\lambda$ with 
$\lambda\leq\kappa$ (see \cite[Corollary~2.6.5]{CB}). Hence we can assume without 
loss of generality that $L=\mathbb{R}^\lambda$ and that $S$ is a vector subspace of $L$.

Fix a countable $\pi$-network $\mathcal{V} = \{V_n: n\in\omega\}$ for $L$ 
as in Lemma~\ref{prop:pinet}. Since $\ldim V_0\geq\cont$
we can find an element $x_0\in V_0\setminus S$. By induction,
if we have defined elements $x_0,\ldots,x_{n-1}$ in $L$ with $x_i\in V_i$ for
each $i\leq n-1$, then we denote by  $S_{n}$ the minimal vector subspace 
of $L$ containing $S$ and the elements $x_0,\ldots,x_{n-1}$. It is clear that 
$\ldim(S_{n})<\cont$, so there exists $x_{n}\in V_{n}\setminus S_{n}$.
The set $X_n = \{x_0,\ldots,x_{n}\}$ generates a vector subspace $M_n$ of $L$.
Finally, we define the set $X=\{x_n: n\in\omega\}$. Let $M$ be the linear subspace 
of $L$ generated by $X$. Then $M$ is dense in $L$ since $X$ is dense in $L$.
Also, $M$ and $S$ have trivial intersection since each $M_{n}$ has trivial intersection 
with $S$, and $M=\bigcup_{n\in\omega} M_n$.
\end{proof}

Now we are in a position to present the main result of this section.

\begin{theorem}\label{Th:LCS}
Assume that $2^{\omega_1}=\cont$. Then there exist pseudocomplete locally convex 
spaces $K$ and $L$ such that all closed vector subspaces of $K$ and $L$ are separable, 
but the product $K\times L$ contains a closed non-separable $\sigma$-compact vector 
subspace $M$. 
\end{theorem}

\begin{proof}
Our construction of $K,L$ and $M$ is similar in spirit to the one presented in the proof 
of Theorem~\ref{Th:Gr}. Let $\mathit{\Pi}=\mathbb{R}^{\omega_1}$ be the product space 
with the usual Tychonoff topology. Clearly $\mathit\Pi$ is a lcs of weight $\omega_1$.
Let $M_0=\sigma\mathit\Pi$ be the $\sigma$-product lying in $\mathit\Pi$, i.e.~the vector 
subspace of $\mathit\Pi$ consisting of all elements which differ from zero element in at 
most finitely many coordinates. One can easily verify that the topological space $M_0$ 
is a $\sigma$-compact and non-separable.

Our aim is to construct two linear subspaces $K$ and $L$ of $\mathit\Pi$ 
satisfying $K\cap L=M_0$. Let $\Delta=\{(x,x): x\in\mathit\Pi\}$ be the diagonal 
in $\mathit\Pi\times\mathit\Pi$. Then $M_0$ is naturally identified, algebraically 
and topologically, with the closed subspace $M=\Delta\cap (K\times L)$ of $K\times L$. 
Since the space $\mathit\Pi$ is locally convex, so are the linear subspaces $K$ and 
$L$ of $\mathit\Pi$.

According to \cite[Corollary~2.6.5]{CB} every non-trivial closed vector subspace 
$C$ of $\mathit\Pi=\mathbb{R}^{\omega_1}$ is isomorphic to $\mathbb{R}^\lambda$, 
where $1\leq \lambda\leq\omega_1$, so $C$ is separable. Since $w(\mathit\Pi)=\omega_1$, 
the assumption $2^{\omega_1}=\cont$ implies that the space $\mathit\Pi$ contains at most 
$\cont$ closed subsets. We enumerate all closed infinite-dimensional vector subspaces 
of $\mathit\Pi$, say, $\{C_\alpha: \alpha<\cont\}$. Also, let $\{b_\alpha: \alpha<\cont\}$ be an 
enumeration of the set $\bigcup\{\mathbb{R}^A: A\subset\omega_1,\ 1\leq |A|\leq\omega \}$.

We put $K_0=L_0=M_0$. Then $\ldim(K_0)=\ldim(L_0)=\omega_1$. Following the lines 
of the proof of Theorem~\ref{Th:Gr} and applying Proposition~\ref{pro:Lin}, one can 
construct families $\{K_\alpha: \alpha<\cont\}$ and $\{L_\alpha: \alpha<\cont\}$ of vector
subspaces of $\mathit\Pi$ satisfying the following conditions for all ordinals $\alpha,\beta$ 
with $\alpha<\beta<\cont$:
\begin{enumerate}
\item[(i)]   $K_\alpha\subset K_\beta$ and $L_\alpha\subset L_\beta$;
\item[(ii)]  $\ldim(K_\alpha)\leq |\alpha+1|\cdot\omega_1$ and $\ldim(L_\alpha)\leq |\alpha+1|   
                \cdot\omega_1$; 
\item[(iii)] $K_\alpha\cap L_\alpha=M_0$;
\item[(iv)] $b_\alpha\in \pi_{\beta}(K_\beta)\cap \pi_{\beta}(L_\beta)$, where $\pi_\beta$
                is the projection of $\mathbb{R}^{\omega_1}$ onto $\mathbb{R}^\beta$;
\item[(v)]  both $K_\beta\cap C_\alpha$ and $L_\beta\cap C_\alpha$ contain a dense 
                separable subspace of $C_\alpha$.
\end{enumerate}
Once the families $\{K_\alpha: \alpha<\cont\}$ and $\{L_\alpha: \alpha<\cont\}$
have been defined, we put $K=\bigcup_{\alpha<\cont} K_\alpha$ and 
$L=\bigcup_{\alpha<\cont} L_\alpha$. Then, by (i) and (iii), the linear subspaces $K$ 
and $L$ of $\mathit\Pi$ satisfy $K\cap L=M_0$. We claim that all closed vector 
subspaces of $K$ and $L$ are separable. For instance, let $F$ be a closed 
vector subspace of $K$. If $F$ is finite-dimensional, then it is evidently separable. Hence
we can assume that $F$ is infinite-dimensional.  The closure of $F$ in $\mathit\Pi$, say, 
$\overline{F}$ is a closed infinite-dimensional vector subspace of $\mathit\Pi$,
so $\overline{F} =C_\alpha$ for some $\alpha<\cont$. By (v), $K_{\alpha+ 1}\cap C_\alpha$ 
contains a dense separable subspace of $C_\alpha$, say, $Q$. Hence $Q$ is dense in 
$F=\overline{F}\cap K=K\cap C_\alpha$, i.e.~$F$ is separable. The same argument
shows that all closed vector subspaces of $L$ are separable as well. 

It follows from (iv) that $\pi_\beta(K)=\pi_\beta(L)=\mathbb{R}^\beta$ for each 
$\beta<\omega_1$. Therefore, by Theorem~\ref{projections2}, $K$ and $L$ are 
pseudocomplete spaces, as claimed. 
\end{proof}

\begin{remark}
{\rm
Since $M_0$ is dense in $\mathit\Pi$, every compact subset of $M_0$ is nowhere dense in 
$\mathit\Pi$ and in $M_0$. Hence the closed vector subspace $M$ of the product $K\times L$ 
in Theorem~\ref{Th:LCS} is not Baire. Similarly, the closed subgroup $\Delta\cap (G\times H)$ 
of $G\times H$ in Theorem~\ref{Th:Gr} is not Baire either.} 
\end{remark}

We finish the article with a question that remains open. 

\begin{problem}\label{Prob:L1}
Are Theorems~\ref{Th:Gr} and~\ref{Th:LCS} valid in ZFC alone?
\end{problem}



\begin{thebibliography}{99}


\bibitem{AT} Alexander V.~Arhangel'skii and Mikhail G.~Tkachenko,
\newblock \emph{Topological Groups and Related Structures}, 
\newblock Atlantis Series in Mathematics, Vol.~I, Atlantis Press and 
World Scientific, Paris--Amsterdam, 2008.

\bibitem{CB} Pedro~P\'erez Carreras and Jos\'e~Bonet,
\newblock \emph{Barreled Locally Convex Spaces},
North-Holland Math. Studies, Notas de Matem\'atica (113),
L.~Nachbin, ed. Elsevier Science Publ. B.V., 1987.

\bibitem{ComItz} Wistar W.~Comfort and Gerald L.~Itzkowitz,
\newblock Density characters in topological groups,
\newblock \emph{Math. Ann.} \textbf{226} (1977), 223--227.

\bibitem{DikUsp} Dikran N.~Dikranjan and Vladimir V.~Uspenskij,
\newblock Categorically compact topological groups,
\newblock \emph{J. Pure Appl. Algebra} \textbf{126} (1--3)(1998), 149--168.

\bibitem{Domanski} Pawel Doma\'nski,
\newblock Nonseparable closed subspaces in separable products of topological 
vector spaces, and $q$-minimality,
\newblock \emph{Arch. Math.}  \textbf{41} (1983), 270--275.

\bibitem{Eng89} Richard Engelking, 
\newblock \textit{General Topology}, 
\newblock Heldermann Verlag, Berlin 1989.


\bibitem{Itz} Gerald L.~Itzkowitz, 
\newblock On the density character of compact topological groups, 
\newblock \textit{Fund. Math.} \textbf{75} (1972), 201--203. 

\bibitem{KOO} Anton A.~Klyachko, Alexander Yu.~Ol'shanskii, Denis V.~Osin,
\newblock On topologizable and non-topologizable groups,
\newblock \emph {Topology Appl.} \textbf{160} (2013), 2104--2120. 

\bibitem{KLM} Jerzy K{\c{a}}kol, Arkady G.~Leiderman and Sidney A.~Morris,
\newblock Nonseparable closed vector subspaces of separable topological vector spaces,
\newblock \emph{Monatsh. Math.}, 9~pp. 
DOI: 10.1007/s00605-016-0876-2.

\bibitem{Lacey} H. Elton Lacey,
\newblock The Hamel dimension of any infinite dimensional separable 
Banach space is $\cont$,
\newblock \emph{Amer. Math. Monthly}  \textbf{80} (3) (1973), 298.

\bibitem{LMT} Arkady G.~Leiderman, Sidney A.~Morris, and Mikhail G.~Tkachenko,
\newblock Density Character of Subgroups of Topological Groups,  
\newblock \emph{Trans. Amer. Math. Soc.}, 20~pp. 
DOI: 10.1090/tran/6832. 

\bibitem{Lohman} Robert H.~Lohman and Wilbur J.~Stiles,
\newblock On separability in linear topological spaces,
\newblock \emph{Proc. Amer. Math. Soc.}  \textbf{42} (1974), 236--237.

\bibitem{Lukacs} G\'abor Luk\'acs,
\newblock Hereditarily non-topologizable groups,
\newblock \emph{Topology Proc.}  \textbf{33} (2009), 269--275.

\bibitem{Mill} Jan van~Mill and Vladimir V.~Tkachuk,
\newblock Every $k$-separable \v{C}ech-complete space is subcompact,
\newblock \emph{RACSAM} \textbf{109} (2015), 65--71.

\bibitem{Morris} Sidney A.~Morris,
\emph{Pontryagin duality and the structure of locally compact abelian groups}, 
\newblock Cambridge University Press, Cambridge--London--New York--Melbourne, 1977.


\bibitem{Ox} John C.~Oxtoby, 
\newblock Cartesian products of Baire spaces, 
\newblock \emph{Fund. Math.} \textbf{49} (1961), 157--166.


\bibitem{Roitman} Judi Roitman,
\newblock Basic $S$ and $L$,
\newblock In: \emph{Handbook of Set Theoretic Topology} (K. Kunen and J. Vaughan, eds.),
\newblock Chapter~7, pp.~295--326. North Holland, Amsterdam, 1984. 

\bibitem{Rudin} Walter Rudin, 
\newblock \emph{Functional Analysis}, 
\newblock McGraw-Hill, 1991.

\bibitem{Tk83} Mikhail G.~Tkachenko, 
\newblock The notion of o-tightness and $C$-embedded subsets of products, 
\newblock \textit{Topol. Appl.} \textbf{15} (1) (1983), 93--98.




\end{thebibliography}
\end{document}